\documentclass[12pt]{amsart}
\usepackage[top=30truemm,bottom=30truemm,left=25truemm,right=25truemm]{geometry}
\usepackage{txfonts}
\usepackage{mathrsfs}

\usepackage{bm}
\usepackage{amsfonts,amssymb}
\usepackage{dsfont}
\usepackage{extarrows}
\usepackage{amsmath}
\usepackage{mathrsfs}
\usepackage{enumerate}
\usepackage{amscd}
\usepackage[all]{xy}
\usepackage[pagebackref,colorlinks]{hyperref}

\theoremstyle{plain} %text of this environment is typesetted in italics

\theoremstyle{definition} %text of this environment is typesetted in roman letters

\newtheorem{thm}{Theorem}[section]

\newtheorem{lem}[thm]{Lemma}

\theoremstyle{definition}
\newtheorem{defn}{Definition}[section]

\theoremstyle{remark}
\newtheorem{rem}{Remark}[section]

\newcommand{\be}{\begin{equation}}
\newcommand{\ee}{\end{equation}}
\newcommand{\bea}{\begin{eqnarray}}
\newcommand{\eea}{\end{eqnarray}}
\newcommand{\ben}{\begin{eqnarray*}}
	\newcommand{\een}{\end{eqnarray*}}
\newcommand{\bt}{\begin{split}}
	\newcommand{\et}{\end{split}}
\newcommand{\bet}{\begin{equation}}
\newcommand{\mc}{\mathbb{C}}
\newcommand{\mr}{\mathbb{R}}
\newcommand{\ra}{\rightarrow}
\newcommand{\ddbar}{\partial \bar{\partial}}
\newcommand{\dbar}{\bar{\partial}}

\begin{document}

\title[Characterizations of plurisubharmonic functions]
{Characterizations of plurisubharmonic functions}

\author[F. Deng]{Fusheng Deng}
\address{Fusheng Deng: \ School of Mathematical Sciences, University of Chinese Academy of Sciences\\ Beijing 100049, P. R. China}
\email{fshdeng@ucas.ac.cn}
\author[J. Ning]{Jiafu Ning}%\textsuperscript{*} }
\address{Jiafu Ning: \ Department of Mathematics, Central South University, Changsha, Hunan 410083\\ P. R. China.}
\email{jfning@csu.edu.cn}
\author[Z. Wang]{Zhiwei Wang}%\textsuperscript{*}}
\address{ Zhiwei Wang: \ School
	of Mathematical Sciences\\Beijing Normal University\\Beijing\\ 100875\\ P. R. China}
\email{zhiwei@bnu.edu.cn}

%32U05=pluripotential theory-Plurisubharmonic functions and generalizations
% 32W05=differential operators in SCV-$\overline\partial$ and $\overline\partial$-Neumann operators
%32A70=hol functions SCV-Functional analysis techniques 
%32T99=Pseudoconvex domains-None of the above, but in this section

%\thanks{(*) The second author and the third author are both corresponding authors.}

%%%%%%%%%%%%%%%%%%%%%%%%%%%%%%%%%%%%%%%%%%%%%%%%%%%%%%
\begin{abstract}
We give characterizations of (quasi-)plurisubharmonic functions
in terms of  $L^p$-estimates of $\bar\partial$ and $L^p$-extensions of holomorphic functions.
\end{abstract}

%%%%%%%%%%%%%%%%%%%%%%%%%%%%%%%%%%%%%%%%%%%%%%%%%%%%%%

\maketitle

%%%%%%%%%%%%%%%%%%%%%%%%%%%%%%%%%%%%%%%%%%%%%%%%%%%%%%
\section{Introduction}
In the fundamental work \cite{Hor65},
H\"{o}rmander established a systematic $L^2$-theory for $\bar\partial$ operator.
In \cite{OT1}, Ohsawa and Takegoshi proved an extension theorem for $L^2$ holomorphic functions,
now known as Ohsawa-Takegoshi extension theorem.
H\"{o}rmander's $L^2$-theory for $\bar\partial$ and Ohsawa-Takegoshi extension theorem
are of fundamental importance in several complex variables and
have terrific applications in algebraic geometry.

%which also have many important applications
%in the areas mentioned above.
%While in the very beginning H\"{o}rmander has proven an optimal version
%of the $L^2$-estimate of $\bar\partial$-equation,
%the optimal estimate for the Ohsawa-Takegoshi extension theorem was just proved
%recently by Blocki \cite{Bl} and Guan-Zuou \cite{GZh12,GZh15d}.

In both H\"{o}rmander's $L^2$-theory for $\bar\partial$ and Ohsawa-Takegoshi extension theorem,
plurisubharmonic functions are used as weights and play a key role.
It is natural to ask if plurisubharmonic functions are the only choice for weights in both theories?
The aim of the present paper is to answer this question affirmatively and give characterizations of (quasi-)plurisubharmonic functions
in terms of $L^p$-estimates of $\bar\partial$ and $L^p$-extensions of holomorphic functions,
which can be roughly understood as converses of H\"{o}rmander's $L^2$ estimate
of $\bar\partial$ and Ohsawa-Takegoshi extension theorem.
To state the results, we first introduce some notions.

\begin{defn}\label{def:L^p peroperty}
Let $\phi:D\ra [-\infty, +\infty)$ be an upper semi-continuous function on a domain $D$ in $\mc^n$.
We say that:
\begin{itemize}
\item[(1)]
$\phi$ satisfies \emph{the optimal $L^p$-estimate property} if for any $\bar\partial$-closed smooth $(0,1)$-form $f$ on $D$ with compact support
and any smooth strictly plurisubharmonic function $\psi$ on $D$, the equation $\bar\partial u=f$ can be solved on $D$
with the estimate
$$\int_D|u|^pe^{-\phi-\psi}\leq \int_D|f|_{i\partial\bar\partial\psi}^pe^{-\phi-\psi}$$
provided that the right hand side is finite.
\item[(2)]
$\phi$ satisfies \emph{the multiple coarse $L^p$-estimate property} if for any $m\geq 1$,
any $\bar\partial$-closed smooth $(0,1)$-form $f$ on $D$ with compact support,
and any smooth strictly plurisubharmonic function $\psi$ on $D$,
the equation $\bar\partial u=f$ can be solved on $D$ with the estimate
$$\int_D|u|^pe^{-m\phi-\psi}\leq C_m\int_D|f|_{i\partial\bar\partial\psi}^pe^{-m\phi-\psi},$$
provided that the R.H.S is finite, where $C_m$ are constants such that
$\lim\limits_{m\ra\infty}\log C_m/m=0$.
\item[(3)]
$\phi$ satisfies \emph{the optimal $L^p$-extension property} if for any $z\in D$ with $\phi(z)\neq -\infty$,
and any holomorphic cylinder $P$ with $z+P\subset D$,
there is a holomorphic function $f$ on $z+P$ such that $f(z)=1$ and
$$\frac{1}{\mu(P)}\int_{z+P}|f|^pe^{-\phi}\leq e^{-\phi(z)},$$
where $\mu(P)$ is the volume of $P$ with respect to the Lebesgue measure.
(Here by a holomorphic cylinder we mean a domain of the form $A(P_{r,s})$
for some $A\in U(n)$ and $r,s>0$,
with $P_{r,s}=\{(z_1,z_2,\cdots,z_n):|z_1|^2<r^2,|z_2|^2+\cdots+|z_n|^2<s^2\}$.)
\item[(4)]
$\phi$ satisfies \emph{the multiple coarse $L^p$-extension property} if for any $m\geq 1$,
any $z\in D$ with $\phi(z)\neq -\infty$,
there is a holomorphic function $f$ on $D$ such that $f(z)=1$ and
$$\int_D|f|^pe^{-m\phi}\leq C_m e^{-m\phi(z)},$$
where $C_m$ are constants such that $\lim\limits_{m\ra\infty}\log C_m/m=0$.
\end{itemize}
\end{defn}

A related but different form of \emph{the optimal $L^p$-extension property}
for $p=2$ was introduced in \cite{HPS16},
\emph{the multiple coarse $L^p$-extension property} was introduced in \cite{DWZZ1},
and \emph{the multiple coarse $L^p$-estimate property} was introduced and named
as the twisted H\"{o}rmander condition in \cite{HI} in the case that $p=2$,
with a different form in dimension one case introduced in \cite{Bob98}.

If $\phi$ is plurisubharmonic, $p=2$, and $D$ is a bounded pseudoconvex domain,
then $\phi$ satisfies:
\begin{itemize}
\item \emph{the optimal $L^2$-estimate property} and \emph{the multiple coarse $L^2$-estimate property}
by H\"{o}rmander's $L^2$-estimate of $\bar\partial$
(c.f. \cite[Theorem 1.6.4]{Ber} for an appropriate formulation);
\item \emph{the multiple coarse $L^2$-extension property} by Ohsawa-Takegoshi \cite{OT1}; and
\item \emph{the optimal $L^2$-extension property} by Blocki \cite{Bl} and Guan-Zhou \cite{GZh12, GZh15d}.
\end{itemize}

In the present paper, we will establish the converses of the above results,
namely, we will show that $\phi$ must be plurisubharmonic if it satisfies
one of the properties in Definition \ref{def:L^p peroperty}.
The precise results are formulated as several theorems as follows.

\begin{thm}\label{thm:sharp estimate}
Let $D$ be a domain in $\mathbb{C}^n$, $\phi\in \mathcal{C}^2(D)$.
If $\phi$ satisfies \emph{the optimal $L^2$-estimate property},
then $\phi$ is plursubharmonic on $D$.
\end{thm}

In fact, Theorem \ref{thm:sharp estimate} can be strengthened to characterize quasi-plurisubharmonicity as follows.

\begin{thm}\label{thm:omega-sharp estimate}
Let $D$ be a domain in $\mathbb{C}^n$, $\phi\in \mathcal{C}^2(D)$
and $\omega$ be a continuous real $(1,1)$-form on $D$.
If for any $\bar\partial$-closed smooth $(0,1)$-form $f$ on $D$ with compact support
and any smooth strictly plurisubharmonic function $\psi$ on $D$ with $i\partial\bar\partial\psi+\omega>0$ on $\text{supp} f$,
the equation $\bar\partial u=f$ can be solved on $D$ with the estimate
$$\int_D|u|^2e^{-\phi-\psi}\leq \int_D|f|_{i\partial\bar\partial\psi+\omega}^2e^{-\phi-\psi}$$
provided that the right hand side is finite.
Then $i\partial\bar\partial\phi\geq\omega$ on $D$.
\end{thm}

We prove Theorem \ref{thm:omega-sharp estimate} by connecting $\partial\bar\partial\phi$
with the optimal $L^2$-estimate property via a Bochner type identity,
and then using a localization technique to produce a contradiction if
$i\partial\bar\partial\phi\geq\omega$ is assumed to be not true.

%\begin{rem}
%It is worth to mention that Theorem \ref{thm:sharp estimate} can be used to characterize strict
%plurisubharmonicity. Namely, let $D$ be a domain in $\mathbb{C}^n$, $\phi\in \mathcal{C}^2(D)$
%and $\phi_0\in \mathcal{C}^2(D)$ be plurisubharmonic.
%If for any $\bar\partial$-closed smooth $(0,1)$-form $f$ on $D$ with compact support
%and any smooth strictly plurisubharmonic function $\psi$ on $D$, the equation $\bar\partial u=f$ can be solved on $D$
%with the estimate
%$$\int_D|u|^pe^{-\phi-\psi}\leq \int_D|f|_{i\partial\bar\partial(\psi+\phi_0)}^pe^{-\phi-\psi}$$
%provided that the right hand side is finite.
%Then $\phi-\phi_0$ is plurisubharmonic.
%\end{rem}

\begin{thm}\label{thm:coarse estimate}
Let $D$ be a domain in $\mc^n$ and $\phi:D\ra \mr$ be a continuous function.
%Assume that $Pole(\phi)$ is closed in $D$ and $\phi|_{D\backslash Pole(\phi)}$ is continuous.
If $\phi$ satisfies \emph{the multiple coarse $L^p$-estimate property} for some $p>1$,
then $\phi$ is plursubharmonic on $D$.
\end{thm}
In connection with Theorem \ref{thm:coarse estimate}, we remark that
Berndtsson has proved the following result \cite[Proposition 2.2]{Ber}:
let $D$ be a domain in $\mc$ and $\phi$ be a continuous function on $D$,
such that for any $m\geq 1$ and any $f\in C^\infty_c(D)$
we can solve the equation $\bar\partial u=fd\bar z$ with the estimate
$$\int_D|u|^2e^{-m\phi}\leq C\int_D|f|^2e^{-m\phi}$$
with $C$ a uniform constant independent of $m$,
then $\phi$ is subharmonic.
The method of Berndtsson depends on the fact that the $(0,1)$-Dolbeault cohomology
of $\mc$ with compact support does not vanish, and seems difficult to generalize to higher dimensions.
The case of Theorem \ref{thm:coarse estimate} that $p=2$ and $\phi$ is locally H\"{o}lder continuous for general dimensions was proved in \cite{HI},
by showing that the multiple coarse $L^2$-estimate property implies
the multiple coarse $L^2$-extension property, and then applying Theorem \ref{thm:coarse estension}
in the following.
The question about if the local H\"{o}lder continuity can be removed or not is proposed in \cite{HI}.
Theorem \ref{thm:coarse estimate} answers this question affirmatively and is proved by
modifying the idea in \cite{HI}.

\begin{thm}\label{thm:sharp extension}
Let $\phi:D\ra [-\infty, +\infty)$ be an upper semi-continuous function on a domain $D$ in $\mc^n$.
If $\phi$ satisfies \emph{the optimal $L^p$-extension property} for some $p>0$,
then $\phi$ is plurisubharmonic on $D$.
\end{thm}

Theorem \ref{thm:sharp extension} for $n=1$ is essentially contained in \cite{GZh15d}
where it is shown that the optimal estimate of the Ohsawa-Takegoshi extension theorem
implies Berndtsson's plurisubharmonic variation of relative Bergman kernels \cite{Bob06, Bob09}.
The general case is proved by modifying Guan-Zhou's method.

\begin{thm}\label{thm:coarse estension}
Let $\phi:D\ra [-\infty, +\infty)$ be an upper semi-continuous function on a domain $D$ in $\mc^n$.
If $\phi$ satisfies \emph{the multiple coarse $L^p$-extension property} for some $p>0$,
then $\phi$ is plursubharmonic on $D$.
\end{thm}

Theorem \ref{thm:coarse estension} was originally proved in \cite{DWZZ1},
where the method is motivated by Demailly's idea on the regularization of plurisubharmonic functions  \cite{Dem92}.
In the present paper, we give a new proof of it based on Guan-Zhou's method mentioned above.

\begin{rem}
Note that plurisubharmonicity involves exact inequalities.
It is more or less reasonable to expect that sharp estimates for $\bar\partial$
(namely the optimal $L^p$-estimate property and the optimal $L^p$-extension property)
could imply the plurisubharmonicity of $\phi$.
However, it is difficult to expect that coarse estimates can encode plurisubharmonicity.
The point in Theorem \ref{thm:coarse estimate} and Theorem \ref{thm:coarse estension}
is that we have to consider powers of the trivial line bundle with product metrics $e^{-m\phi}$,
and then there is a procedure similar to taking $m$-th roots of $C_m$ and
finally we get the exact number $1=\lim\sqrt[m]{C_m}$ in the limit.
This is the main observation in \cite{DWZZ1}.
\end{rem}
\begin{rem}
Continuity assumption in Theorem \ref{thm:coarse estimate}
is in some sense optimal since the weight $\phi$ appears in integrations in both sides
of the involved estimates.
This condition can be weakened to semi-continuity in Theorem \ref{thm:sharp extension}
and Theorem \ref{thm:coarse estension} since the related estimates involves pointwise
evaluations of $\phi$.
Therefore it is natural to ask whether the regularity condition on $\phi$ in Theorem \ref{thm:omega-sharp estimate}
can be weakened to being continuous?
\end{rem}
The above theorems, combining with $L^2$-theory of $\bar\partial$,
imply that the four properties in Definition \ref{def:L^p peroperty}
are essentially equivalent for $\phi$ at least for the case that $p=2$.
Considering this conclusion is built on some heavy theories,
it seems interesting to find more straightforward methods
to establish the equivalence of these properties.

We now recall some background and geometric meaning of plurisubharmonic functions.

A plurisubharmonic function on an open set in $\mc^n$ is an upper semi-continuous function
with values in $[-\infty,+\infty)$ whose restriction to any complex line is subharmonic.
It turns out that the definition does not depend on the linear structure on $\mc^n$ and
is invariant under holomorphic coordinate changes,
hence plurisubharmonic functions can be defined on complex manifolds (or even complex spaces).

Plurisubharmonic functions play very important roles in several complex variables and complex geometry.
Geometrically, a basic fact is that plurisubharmonic functions are connected to positivity
of curvatures of Hermitian holomorphic vector bundles.

For the line bundle case, let $L$ be a holomorphic line bundle over a complex manifold $X$ and $h$ be a Hermitian metric on $L$.
Let $e$ be a holomorphic local frame of $L$ on some open set $U\subset X$.
Then $\|e\|^2$ is a positive function on $U$ and thus can be written as $e^{-\phi}$ for some
smooth function $\phi$ on $U$.
Then the curvature of $(L,h)$ on $U$ can be written as $i\partial\bar\partial\phi$.
Therefore $(L,h)$ is positive (resp. semi-positive) if and only if its local weights $\phi$
are strictly plurisubharmonic (resp. plurisubharmonic).
In applications, it is also very important to allow $\phi$ has singularities,
we then get the so called singular Hermitian metrics on $L$ (see \cite{Dem12}).

The above characterization of positivity of Hermitian line bundles in terms of plurisubharmonicity
can be generalized to Hermitian holomorphic vector bundles of higher rank.
Let $(E,h)$ be a Hermitian holomorphic vector bundle over $X$.
Then we can define some positivity-called the Griffiths positivity-of the curvature
of the Chern connection on $E$.
It turns out that $(E,h)$ is Griffiths negative (resp. semi-negative) if for any local nonvanishing
holomorphic section $s$ of $E$, $\log\|s\|$ is a strictly plurisubharmonic (resp. plurisubharmonic) function,
and $(E,h)$ is Griffiths positive (resp. semi-positive) if and only if its dual bundle $E^*$ with the dual metric $h^*$
is Griffiths negative (resp. semi-negative).

From the above discussion, we can get the conclusion that geometrically
plurisubharmonicity is in some sense equivalent to Griffiths positivity of Hermitian holomorphic vector bundles.

Theorem \ref{thm:sharp estimate}-\ref{thm:coarse estension}
can be viewed as local characterizations of positive Hermitian line bundles.
These results can be generalized to
Hermitian holomorphic vector bundles of higher rank.
However, for highlighting the key ideas in our method,
we will not discuss vector bundles of higher rank in the present paper.

\textbf{Acknowledgement:}
The authors are very grateful to Professor
Xiangyu Zhou, their former advisor, for valuable discussions on related topics.
The authors are partially supported by NSFC grants.

\section{Characterizations of plurisubharmonic functions in terms of $L^p$-estimates of $\bar\partial$}
\subsection{In terms of the optimal $L^2$-estimate property}
The aim of this subsection is to prove Theorem \ref{thm:omega-sharp estimate}.
For convenience, we restate it here.

\begin{thm}[= Theorem 1.2]\label{thm:sharp estimate text}
Let $D$ be a domain in $\mathbb{C}^n$, $\phi\in \mathcal{C}^2(D)$
and $\omega$ be a continuous real $(1,1)$-form on $D$.
If for any $\bar\partial$-closed smooth $(0,1)$-form $f$ on $D$ with compact support
and any smooth strictly plurisubharmonic function $\psi$ on $D$ with $i\partial\bar\partial\psi+\omega>0$ on $\text{supp} f$,
the equation $\bar\partial u=f$ can be solved on $D$ with the estimate
$$\int_D|u|^2e^{-\phi-\psi}\leq \int_D|f|_{i\partial\bar\partial\psi+\omega}^2e^{-\phi-\psi}$$
provided that the right hand side is finite.
Then $i\partial\bar\partial\phi\geq\omega$ on $D$.
\end{thm}

We need the following lemma.

\begin{lem}[{\cite[Proposition 2.1.2]{Hor65}}]\label{lem1}
Let $ D$ be a domain in $\mathbb{C}^n$, $\phi\in\mathcal{C}^2( D)$. For any $\alpha\in\mathcal{D}_{0,1}( D)$, $\alpha=\sum_{j=1}^{n}\alpha_jd\bar{z}_j$,
we have
\begin{equation}\label{base}
\int_ D \sum_{j,k=1}^n\frac{\partial^2\phi}{\partial z_j\partial\bar{z}_k}\alpha_j\overline{\alpha}_k e^{-\phi}+
\int_ D\sum_{j,k=1}^n|\frac{\partial\alpha_j}{\partial\bar{z}_k}|^2e^{-\phi}=
\int_{ D}|\bar{\partial}\alpha|^2e^{-\phi}+\int_ D|\bar\partial^*_\phi\alpha|^2e^{-\phi}.
\end{equation}
\end{lem}
We now give the proof of Theorem \ref{thm:sharp estimate text}.
\begin{proof}
We give the proof of Theorem \ref{thm:sharp estimate text} in the case that $\omega$ is $\mathcal C^1$,
and the general case follows the proof by an approximation argument.

Let $\psi$ be any smooth strictly plurisubharmonic function on $D$, and $$\omega=i\sum_{j,k=1}^{n}g_{j\bar k}dz_j\wedge d\bar{z}_k.$$
By assumption, we can solve the equation $\bar{\partial}u=f$  for any $\bar\partial $-closed $f\in \mathcal D_{0,1}(D)$, with the estimate
$$\int_ D|u|^2e^{-(\phi+\psi)}\leq\int_ D|f|^2_{i\partial\bar{\partial}\psi+\omega}e^{-(\phi+\psi)}.$$
For any $\alpha\in\mathcal{D}_{0,1}( D)$, we have
\begin{align*}
|(\alpha,f)_{\phi+\psi}|&=|(\alpha,\bar{\partial}u)_{\phi+\psi}|\\
&=|(\bar{\partial}_{\phi+\psi}^*\alpha,u)_{\phi+\psi}|\\
&\leq||u||_{\phi+\psi}||\bar{\partial}_{\phi+\psi}^*\alpha||_{\phi+\psi}.
\end{align*}
From  Lemma \ref{lem1}, we obtain
\begin{equation}\label{eq1}
\begin{split}
&|(\alpha,f)_{\phi+\psi}|^2\\
\leq& \int_ D|f|^2_{i\partial\bar{\partial}\psi+\omega}e^{-(\phi+\psi)}\\
&\times\left(\int_ D \sum_{j,k=1}^n\frac{\partial^2(\phi+\psi)}{\partial z_j\partial\bar{z}_k}\alpha_j\bar{\alpha}_k e^{-(\phi+\psi)}+
\int_ D\sum_{j,k=1}^n|\frac{\partial\alpha_j}{\partial\bar{z}_k}|^2e^{-(\phi+\psi)}-\int_{ D}|\bar{\partial}\alpha|^2e^{-(\phi+\psi)}\right)\\
\leq& \int_ D|f|^2_{i\partial\bar{\partial}\psi+\omega}e^{-(\phi+\psi)}\times\left(\int_ D \sum_{j,k=1}^n\frac{\partial^2(\phi+\psi)}{\partial z_j\partial\bar{z}_k}\alpha_j\bar{\alpha}_k e^{-(\phi+\psi)}+
\int_ D\sum_{j,k=1}^n|\frac{\partial\alpha_j}{\partial\bar{z}_k}|^2e^{-(\phi+\psi)}\right).
\end{split}
\end{equation}

Let $f=\sum_{j=1}^{n}f_jd\bar{z}_j$, set
$$(\alpha_1,\alpha_2,\cdots,\alpha_n)=(f_1,f_2,\cdots,f_n)\left( \frac{\partial^2\psi}{\partial z_j\partial\bar{z}_k}+g_{j\bar k}\right)_{n\times n}^{-1},$$
i.e.,
$$(f_1,f_2,\cdots,f_n)=(\alpha_1,\alpha_2,\cdots,\alpha_n)\left( \frac{\partial^2\psi}{\partial z_j\partial\bar{z}_k}+g_{j\bar k}\right)_{n\times n}.$$
Then inequality (\ref{eq1}) becomes
\begin{equation*}
\begin{split}
&\left(\int_ D\sum_{j,k=1}^n  \left(\frac{\partial^2\psi}{\partial z_j\partial\bar{z}_k}+g_{j\bar k}\right)\alpha_j\bar{\alpha}_ke^{-(\phi+\psi)}\right)^2\\
\leq&\int_{ D}\sum_{j,k=1}^n  \left(\frac{\partial^2\psi}{\partial z_j\partial\bar{z}_k}+g_{j\bar k}\right)\alpha_j\bar{\alpha}_ke^{-(\phi+\psi)}\left(\int_ D \sum_{j,k=1}^n \frac{\partial^2(\phi+\psi)}{\partial z_j\partial\bar{z}_k}\alpha_j\bar{\alpha}_k e^{-(\phi+\psi)}+
\int_ D\sum_{j,k=1}^n|\frac{\partial\alpha_j}{\partial\bar{z}_k}|^2e^{-(\phi+\psi)}\right).
\end{split}
\end{equation*}
Therefore, we can get
\begin{equation}\label{eq2}
\int_{ D}\sum_{j,k=1}^n \left(\frac{\partial^2\phi}{\partial z_j\partial\bar{z}_k}-g_{j\bar k}\right)\alpha_j\bar{\alpha}_ke^{-(\phi+\psi)}+ \int_ D\sum_{j,k=1}^n|\frac{\partial\alpha_j}{\partial\bar{z}_k}|^2e^{-(\phi+\psi)}\geq 0.
\end{equation}
We argue by contradiction. Suppose that $i\partial\bar\partial\phi-\omega$ is not a semipositive $(1,1)$-form  on $ D$, then there is $z_0\in D$, $r>0$, a constant $c>0$, and $\xi=(\xi_1,\xi_2,\cdots,\xi_n)\in
\mathbb{C}^n$ with $|\xi|=1$,  such that
$$\sum_{j,k=1}^{n}\left(\frac{\partial^2\phi}{\partial z_j\partial\bar{z}_k}-g_{j\bar k}(z)\right)\xi_j\bar{\xi}_k<-c$$
holds for any $z\in B(z_0,r):=\{z\in \mc^n: |z-z_0|<r\}\subset D$.
We may assume that $z_0=0$, and write $B(0,r)$ as $B_r$.

Choose $\chi\in \mathcal{C}^\infty_c(B_r)$, satisfying $\chi(z)=1$ for $z\in B_{r/2}$.
Let $f=\bar{\partial}\nu$ with
$$\nu(z)=(\sum_{j=1}^{n}\xi_j\bar{z}_j)\chi(z).$$
Then $$f(z)=\sum_{j=1}^{n}\xi_jd\bar{z}_j$$
for $z\in B_{r/2}$.
For $s>0$, set $$\psi_s(z)=s(|z|^2-\frac{r^2}{4}).$$
It is obvious that $i\partial\bar\partial\psi_s+\omega>0$ on $\text{supp} f$ for $s\gg 1$.
As before, set
$$(\alpha^s_1,\alpha^s_2,\cdots,\alpha^s_n)=(f_1,f_2,\cdots,f_n)( \frac{\partial^2\psi_s}{\partial
	z_j\partial\bar{z}_k}+g_{j\bar{k}})^{-1}=\frac{1}{s}(\delta_{jk}+\frac{g_{j\bar k}}{s})^{-1}(f_1,f_2,\cdots,f_n).$$

We now estimate the integrations on the left hand side of \eqref{eq2} on $B_{r/2}$ and $D\backslash B_{r/2}$ separately,
with $\alpha$ and $\psi$ replaced by $\alpha^s$ and $\psi_s$ respectively.
\begin{itemize}
\item
On $B_{r/2}$, as $s\ra +\infty$, we have
$$\alpha^s(z)=\frac{1}{s}\sum_{j=1}^{n}\xi_jd\bar{z}_j+o(\frac{1}{s}),$$
and
$$\frac{\partial\alpha_j^s}{\partial \bar{z}_k}(z)=o(\frac{1}{s}),\ j,k=1,2,\cdots,n.$$
Hence we have
\begin{equation}\label{eq4}
\begin{split}
&s^2\int_{B_{r/2}}\sum_{j,k=1}^n \left(\frac{\partial^2\phi}{\partial z_j\partial\bar{z}_k}-g_{j\bar k}\right)\alpha_j\bar{\alpha}_ke^{-(\phi+\psi_s)}+ s^2\int_{B_{r/2}}\sum_{j,k=1}^n|\frac{\partial\alpha_j}{\partial\bar{z}_k}|^2e^{-(\phi+\psi_s)}\\
&\leq\int_{B_{r/2}}(-c+o(1))e^{-(\phi+\psi_s)}.
\end{split}
\end{equation}

\item
On $D\backslash B_{r/2}$,
since $f$ has compact support,
there is a constant $C>0$, such that $|\alpha^s_j|\leq \frac{C}{s}$
and $|\frac{\partial\alpha^s_j}{\partial \bar{z}_k}|\leq \frac{C}{s}$ hold for $j,k=1,2,\cdots,n$ and $s>0$.
Note also that $\lim_{s\rightarrow+\infty}\psi_s(z)=+\infty$ for $z\in  D\setminus \overline{B}_{r/2}$.
We get
\begin{equation}\label{eq5}
\lim_{s\rightarrow+\infty}\left(s^2\int_{D\setminus B_{r/2}}\sum_{j,k=1}^n \left(\frac{\partial^2\phi}{\partial z_j\partial\bar{z}_k}-g_{j\bar k}\right)\alpha_j\bar{\alpha}_ke^{-(\phi+\psi_s)}+ s^2\int_{D\setminus B_{r/2}}\sum_{j,k=1}^n|\frac{\partial\alpha_j}{\partial\bar{z}_k}|^2e^{-(\phi+\psi_s)}\right)=0.
\end{equation}
\end{itemize}

Note that $\psi_s\leq 0$ on $B_{r/2}$,
combining equalities \eqref{eq4} and \eqref{eq5}
we have,
 $$\int_{ D}\sum_{j,k=1}^n \left(\frac{\partial^2\phi}{\partial z_j\partial\bar{z}_k}-g_{j\bar k}\right)\alpha^s_j\bar{\alpha}^s_ke^{-(\phi+\psi_s)}+ \int_ D\sum_{j,k=1}^n|\frac{\partial\alpha^s_j}{\partial\bar{z}_k}|^2e^{-(\phi+\psi_s)}<0$$ for $s>>1$,
which contradicts to the inequality (\ref{eq2}).

\end{proof}

\begin{rem} From the proof, we can see that in Theorem \ref{thm:sharp estimate text}, $\psi$
can be taken to be of the form $a|z-w|^2+b$, with $a\gg 1$ and $b\in \mr$.
\end{rem}

\begin{rem}
In Theorem \ref{thm:sharp estimate},
we can allow $\phi$ to have poles, with the condition that
$Pol(\phi):=\phi^{-1}(-\infty)$ is closed in $D$,
and $\phi$ is upper semi-continuous on $D$, and is $\mathcal C^2$-smooth and satisfies the optimal $L^2$-estimate property on $D\backslash Pol(\phi)$.
\end{rem}

\subsection{In terms of the multiple coarse $L^p$-estimate property}
The purpose of this subsection is to prove Theorem \ref{thm:coarse estimate},
which is the following

\begin{thm}[= Theorem 1.3]\label{thm:coarse estimate text}
Let $D$ be a domain in $\mc^n$ and $\phi:D\ra \mr$ be a continuous function.
%Assume that $Pole(\phi)$ is closed in $D$ and $\phi|_{D\backslash Pole(\phi)}$ is continuous.
If $\phi$ satisfies \emph{the multiple coarse $L^p$-estimate property} for some $p>1$,
then $\phi$ is plursubharmonic on $D$.
\end{thm}
\begin{proof}
We prove the theorem by modifying the idea in \cite{HI}.
We will show that $( D, \phi)$ satisfies the multiple coarse $L^p$-extension property.
Note that $\phi|_{D'}$ also satisfies the multiple coarse $L^p$-estimate property
for any open set $D'\subset D$, replacing $D$ by a relatively compact open subset of it,
we may assume that $D$ is bounded and $\phi$ is uniformly continuous on $D$.

Fix  an integer $m>0$ and $w\in D$. We will construct a holomorphic function $f \in  \mathcal{O}(D)$ such that $f(w) = 1$ and
$$ \int_{  D} |f|^p e^{-m\phi} \leq C'_m e^{-m\phi(w)},$$
where $C'_m$ is a constant independent of the choice of $w \in   D$ satisfying $\lim_{m\rightarrow\infty}\frac{\log C'_m}{m}=0$.
%Note that, by assumption that $\phi$ is H\"older continuous, we have that $\phi(w) \neq -\infty$.
%In the following, we assume $w = 0$ for the simplicity.

Let $\chi = \chi(t)$ be a smooth function on $\mathbb{R}$, such that
\begin{itemize}
	\item $\chi(t) = 1$ for $t \leq 1/4$,
	\item $\chi(t) = 0$ for $t \geq 1$, and
	\item $|\chi'(t)|\leq 2$ on $\mathbb{R}$.
\end{itemize}
Define a (0,1)-form $\alpha_\epsilon$ by
\begin{align*}
\alpha_\epsilon &:= \dbar \chi\left({|z-w|^2 \over \epsilon^2 } \right)= \chi'\left(\frac{|z-w|^2}{\epsilon^2} \right) \sum_j \frac{z_j-w_j}{\epsilon^2} d\overline{z}_j,
\end{align*}
and set
\begin{equation*}
\psi_{\delta} := |z|^2  + n \log(|z-w|^2 + \delta^2),
\end{equation*}
where $0<\epsilon\leq1$ and $\delta$ are positive parameters.
From the multiple coarse $L^p$-estimate property,
we obtain a smooth function $u_{\epsilon, \delta}$ on $  D$ such that $\dbar u_{\epsilon,\delta} = \alpha_\epsilon$ and
\begin{equation}
\int_{  D} |u_{\epsilon,\delta}|^p e^{-(m\phi + \psi_{\delta})} \leq C_m\int_{  D} |\alpha_\epsilon|^p_{i\ddbar \psi_{\delta}} e^{-(m\phi+ \psi_{\delta})}.\label{eqn:dbar-est}
\end{equation}

Since
\begin{equation*}%{\label{ineq 1}}
\begin{split}%{\label{ineq 1}}
|\alpha_\epsilon|_{\sqrt{-1}\ddbar \psi_{ \delta}}
& = | \chi' (\frac{|z-w|^2}{\epsilon^2} ) | \cdot \frac{1}{\epsilon^2} \cdot
|\sum_j (z_j-w_j) d\overline{z}_j |_{i\ddbar \psi_{\delta}}\\
& \leq  | \chi' (\frac{|z-w|^2}{\epsilon^2} ) | \cdot \frac{1}{\epsilon^2} \cdot |\sum_j (z_j-w_j)
d\overline{z}_j |_{i\ddbar |z|^2}\\
& =| \chi' (\frac{|z-w|^2}{\epsilon^2} )| \cdot \frac{1}{\epsilon^2}|z-w|,
\end{split}
\end{equation*}
the support of $\chi'\left(\frac{|z-w|^2}{\epsilon^2} \right)$ is in $\{1/4 \leq |z-w|^2 /\epsilon^2 \leq 1 \}$, and
$\psi_{\delta}\geq  2n\log|z-w|$, we have that
\begin{equation}\label{ineq t3}
\begin{split}
(\text{RHS of (\ref{eqn:dbar-est})}) &\leq C_m\int_{\{\epsilon^2 / 4 \leq |z-w|^2 \leq \epsilon^2\}}| \chi' (\frac{|z-w|^2}{\epsilon^2} )|^p \frac{1}{\epsilon^{2p}}|z-w|^p e^{-(m\phi + \psi_{\delta})}\\
& \leq C_m\frac{2^p}{\epsilon^{2p}} \int_{\{\epsilon^2 / 4 \leq |z-w|^2 \leq \epsilon^2\}} |z-w|^pe^{-(m\phi + \psi_{ \delta})}\\
&\leq C_m\frac{2^p}{\epsilon^{2p}}\int_{\{\epsilon^2 / 4 \leq |z-w|^2 \leq \epsilon^2\}} \epsilon^p e^{-m\inf_{B(w,\epsilon)}\phi} e^{-2n\log|z-w|}\\
&\leq CC_m\frac{e^{-m\inf_{B(w,\epsilon)}\phi} }{\epsilon^p},
\end{split}
\end{equation}
where $C=2^{p+2n}\mu(B_1)$.

To summarize, we have obtained a smooth function $u_{\epsilon,\delta}$ on $  D$ such that
\begin{itemize}
	\item $\dbar u_{\epsilon,\delta} = \alpha_\epsilon$, and
	\item the following estimate holds:
	\begin{equation}
	\int_{  D} |u_{\epsilon,\delta}|^p e^{-(m\phi + \psi_{\delta})} \leq CC_m\frac{e^{-m\inf_{\mathbf{B}(w,\epsilon)}\phi} }{\epsilon^p}.\label{eqn:matome}
	\end{equation}
\end{itemize}

Note that the weight function $\psi_{\delta}$ is decreasing when $\delta \searrow 0$,  $e^{-\psi_{\delta}}$ is increasing when $\delta \searrow 0$.
Fix $\delta_0>0$. Then, for $\delta < \delta_0$,   we have that
$$\int_{  D} |u_{\epsilon,\delta}|^p e^{-(m\phi + \psi_{\delta_0})} \leq \int_{  D} |u_{\epsilon,\delta}|^p e^{-(m\phi + \psi_{\delta})} \leq CC_m\frac{e^{-m\inf_{B(w,\epsilon)}\phi} }{\epsilon^p}.$$
Thus $\{u_{\epsilon,\delta}\}_{\delta < \delta_0}$ forms a bounded sequence in $L^p(  D, e^{-(m\phi + \delta_0)})$. Note that $p>1$, we can choose a sequence $\{u_{\epsilon,\delta^{(k)}}\}_k$ in $L^p(D, e^{-(m\phi + \delta_0)})$
which weakly converges to some $u_\epsilon\in L^p(D, e^{-(m\phi + \delta_0)})$, satisfying
$$\int_{  D} |u_\epsilon|^p e^{-(m\phi + \psi_{\delta_0})} \leq  CC_m\frac{e^{-m\inf_{B(w,\epsilon)}\phi} }{\epsilon^p}. $$
Repeating this argument for a sequence $\{\delta_j\}$ decreasing to $0$, and by diagonal argument,
we can select a sequence $\{u_{\epsilon,\delta^k}\}_k$ which weakly converges to $u_\epsilon$ in $L^p(D, e^{-(m\phi + \psi_{\delta_j})})$ and $u_\epsilon$ satisfies estimates
$$\int_{  D} |u_\epsilon|^p e^{-(m\phi + \psi_{\delta_j})} \leq  CC_m\frac{e^{-m\inf_{B(w,\epsilon)}\phi} }{\epsilon^p}$$
for all $j$.
By the monotone convergence theorem,
$$\int_{  D} |u_\epsilon|^p e^{-(m\phi + \psi_{0})} \leq  CC_m\frac{e^{-m\inf_{B(w,\epsilon)}\phi} }{\epsilon^p}. $$
Since $\bar\partial$ is weakly continuous, we also have $\dbar u_\epsilon = \alpha_\epsilon.$

Since $\frac{1}{|z-w|^{2n}}$ is not integrable near $w$, $u_\epsilon(w)$ must be $0$.
Let $f_\epsilon := \chi(|z-w|^2/\epsilon^2) - u_\epsilon$. Then $f_\epsilon\in \mathcal O(D)$, $f_\epsilon(0) = 1$ and
\begin{equation}\label{eqn:4}
\begin{split}
\int_{  D} |f_\epsilon|^p e^{-m\phi} &\leq \left(\left(\int_{  D} |\chi({|z-w|^2\over\epsilon^2})|^p e^{-m\phi}\right)^{1/p} + \left(\int_{  D} |u_\epsilon|^p e^{-m\phi}\right)^{1/p}\right)^p\\
&\leq 2^p\left(\int_{  D} |\chi({|z-w|^2\over\epsilon^2})|^p e^{-m\phi}+\int_{  D} |u_\epsilon|^p e^{-m\phi}\right)
%\label{eqn:triangle}
\end{split}
\end{equation}

Since $\chi \leq 1$ and the support of $\chi(|z-w|^2 / \epsilon^2) $ is contained in $\{|z-w|^2 \leq \epsilon^2 \}$ and $0<\epsilon\leq1$, we have
$$\int_{  D} |\chi({|z-w|^2\over\epsilon^2})|^p e^{-m\phi}\leq \mu( B_1)e^{-m \inf_{ B(w,\epsilon)} \phi}.$$
We also have
\begin{align*}
\int_{  D}|u_\epsilon|^p e^{-m\phi} &\leq \sup_{z \in   D} e^{\psi_0(z)} \cdot \int_{  D}
|u_\epsilon|^p e^{-(m\phi+\psi_0) } \\ &\leq
\sup_{z \in   D}e^{\psi_0(z)}  \cdot CC_m\frac{e^{-m\inf_{ B(w,\epsilon)}\phi} }{\epsilon^p} \\ &\leq
C'C_m\frac{e^{-m\inf_{ B(w,\epsilon)}\phi} }{\epsilon^p},
\end{align*}
where $C'$ is a constant depends only on the diameter of $D$. We may assume $C_m\geq1$.
Combining these estimates with \eqref{eqn:4}, we obtain that
\begin{equation*}
\int_{  D} |f_\epsilon|^p e^{-m\phi} \leq C'' C_m\frac{1}{\epsilon^p} e^{-m \inf_{B(w,\epsilon)} \phi},
\end{equation*}
where $C''$ is a constant independent of $m$ and $w$.

Let
$$O_\epsilon=\sup\limits_{z,w\in D, |z-w|\leq\epsilon} |\phi(z)-\phi(w)|.$$
By the uniform continuity of $\phi$, $O_\epsilon$ is finite and goes to 0 as $\epsilon\ra 0$.
Let $\epsilon := 1/m$. We have  $|m\phi(z) - m \phi(w)| \leq m O_{1/m}$ for $|z-w| \leq 1/m$. Then
\begin{align}
\int_{  D} |f_{1/m}|^p e^{-m\phi}
&\leq C''C_m m^{p} e^{-m\phi(w)+m O_{1/m}}\nonumber\\
&= C''C_mm^{p} e^{m O_{1/m}}e^{-m\phi(w)}.\label{eqn:last}
\end{align}
Let $C'_m=C''C_mm^{p} e^{m O_{1/m}}$, we have
$$\frac{\log C'_m}{m}=\frac{\log (C''C_mm^{p})}{m}+ O_{1/m}\to 0,$$
which implies $\phi$ satisfies the multiple coarse $L^p$-extension property on $D$,
and hence $\phi$ is plurisubharmonic by Theorem \ref{thm:coarse estension}.
\end{proof}

\begin{rem}
In Theorem \ref{thm:coarse estimate text},
we can allow $\phi$ to have poles, with the condition that
$Pol(\phi):=\phi^{-1}(-\infty)$ is closed in $D$,
and $\phi$ is upper semi-continuous on $D$, and is continuous and satisfies the multiple coarse $L^p$-estimate property on $D\backslash Pol(\phi)$.
\end{rem}

\section{Characterizations of plurisubharmonic functions in terms of $L^p$-extensions of holomorphic functions}
%\subsection{In terms of the optimal $L^p$-extension property}
%\subsection{In terms of the multiple coarse $L^p$-extension property}
In this section, we discuss characterizations of plurisubharmonic functions
in terms of $L^p$-extensions of holomorphic functions.
The aim is to prove Theorem \ref{thm:sharp extension} and \ref{thm:coarse estension}.
The idea is inspired by Guan-Zhou's method in \cite{GZh15d}.

We first prove a lemma as follows.
\begin{lem}\label{psh}
Let $D\subset\mathbb{C}^n$ be a domain, and $\phi$ be an upper semi-continuous function on $D$.
Then $\phi$ is plurisubharmonic
if and only if for any $z_0\in D$ and any holomorphic cylinder $P$ with with $z_0+P\subset\subset D$,
$$\phi(z_0)\leq \frac{1}{\mu(P)}\int_{z_0+P}\phi.$$
\end{lem}

\begin{proof}
The "only if" part follows easily from the mean value inequality for plurisubharmonic functions.
We now give the proof of the "if" part.

For any point $z_0\in D$, any $\xi\in\mathbb{C}^n$ with $|\xi|=1$, and $r>0$
such that $\{z_0+\zeta_1\xi:|\zeta_1|\leq r_1\}\subset D$.
Choose an orthonormal basis $f_1,f_2,\cdots, f_n$ of $\mathbb{C}^n$ with $f_1=\xi$.
There is $s_0>0$ such that $z_0+\zeta_1\xi+\sum_{j=2}^{n}\zeta_jf_j\in D$ for all $|\zeta_1|\leq r$
and $\sum_{j=2}^{n}|\zeta_j|^2\leq s_0^2$.
Let
$A=(f_1,f_2,\cdots,f_n)$, then $z_0+A(P_{r,s})\subset\subset D$, for $s<s_0$.
By assumption, we have
\begin{equation}\label{lem psh eq1}
\begin{split}
\phi(z_0)&\leq \frac{1}{\mu(A(P_{r,s}))}\int_{z_0+A(P_{r,s})}\phi\\
&=\frac{1}{\mu( P_{r,1})}\int_{P_{r,1}}\phi(z_0+\zeta_1\xi+s\sum_{j=2}^{n}\zeta_jf_j).
\end{split}
\end{equation}
As $\phi$ is upper semi-continuous on $D$ and $\{z_0+\zeta_1\xi:|\zeta_1|\leq r\}$ is compact,
we may assume $\phi\leq 0$
on $z_0+A (P_{r,s_0})$.
Then we have
\begin{equation*}
\begin{split}
\phi(z_0)&\leq \frac{1}{\mu( P_{r,1})}\int_{ P_{r,1}}\limsup\limits_{s\rightarrow 0}\phi(z_0+\zeta_1\xi+s\sum_{j=2}^{n}\zeta_jf_j)\\
&\leq\frac{1}{\mu( P_{r,1})}\int_{ P_{r,1}}\phi(z_0+\zeta_1\xi)\\
&=\frac{1}{\pi r^2}\int_{\{\zeta_1\in\mathbb{C}:|\zeta_1|<r\}}\phi(z_0+\zeta_1\xi),
\end{split}
\end{equation*}
where the first inequality follows from \eqref{lem psh eq1} and Fatou's lemma,
the second inequality follows from the fact that $\phi$ is upper semi-continuous,
and the last equality holds since $\phi(z_0+\zeta_1\xi)$ is independent of $\zeta_2,\cdots, \zeta_n$.
\end{proof}

\begin{thm}[= Theorem 1.4]\label{thm:sharp extension text}
Let $\phi:D\ra [-\infty, +\infty)$ be an upper semi-continuous function on a domain $D$ in $\mc^n$.
If $\phi$ satisfies \emph{the optimal $L^p$-extension property} for some $p>0$,
then $\phi$ is plurisubharmonic on $D$.
\end{thm}
\begin{proof}
Take arbitrary $z_0\in D$, such that $\phi(z_0)\neq -\infty$.
For any holomorphic cylinder $P$ with $z_0+P\subset D$,
by the optimal $L^p$-extension property and taking log, we get that
$$\phi(z_0)\leq -\log\Big(\frac{1}{ \mu(P)}\int_{z_0+P}|f|^pe^{-\phi}\Big).$$
By Jensen's inequality, we have
\begin{align*}
\phi(z_0)&\leq \frac{1}{ \mu(P)}\int_{z_0+P}-\log(|f|^pe^{-\phi})\\
&=\frac{1}{ \mu(P)}\int_{z_0+P}\phi-\frac{1}{ \mu(P)}\int_{z_0+P}p\log|f|.
\end{align*}
Note that $p\log|f|$ is a plurisubharmonic function and $f(z_0)=1$,
by Fubini's theorem, we have $$-\frac{1}{ \mu(P)}\int_{z_0+P}p\log|f|\leq 0.$$
Therefore, we get the desired mean-value inequality,
$$\phi(z_0)\leq\frac{1}{ \mu(P)}\int_{z_0+P}\phi.$$
By Lemma \ref{psh}, $\phi$ is plurisubharmonic on $D$.
\end{proof}

\begin{thm}[= Theorem 1.5]\label{thm:coarse estension text}
Let $\phi:D\ra [-\infty, +\infty)$ be an upper semi-continuous function on a domain $D$ in $\mc^n$.
If $\phi$ satisfies \emph{the multiple coarse $L^p$-extension property} for some $p>0$,
then $\phi$ is plursubharmonic on $D$.
\end{thm}
\begin{proof}
Let $z_0\in D$ such that $\phi(z_0)\neq -\infty$,
and let $P$ be a holomorphic cylinder with $z_0+P\subset\subset D$.
By assumption, for any $m\geq1$,
there is a holomorphic function $f_m$ on $D$, with $f_m(z_0)=1$ and
\begin{equation*}%\label{}
\int_{z_0+P}|f_m|^pe^{-m\phi}\leq\int_D|f_m|^pe^{-m\phi}\leq C_m e^{-m\phi(z)}.
\end{equation*}
Dividing both sides by $\mu(P)$ and taking log, we get
$$\log\frac{1}{\mu(P)}\int_{z_0+P}|f_m|^pe^{-m\phi}\leq \log\frac{C_m}{ \mu(P)}-m\phi(z_0).$$
Then we have
\begin{equation}\label{thm 4 eq1}
\phi(z_0)\leq \frac{\log C_m}{m}-\frac{\log\mu(P)}{m}-\frac{1}{m}\log\frac{1}{\mu(P)}
\int_{z_0+P}|f_m|^pe^{-m\phi}
\end{equation}
By Jensen's inequality and
\begin{equation*}%\label{thm eq2}
\begin{split}
&-\frac{1}{m}\log\left(\frac{1}{\mu(P)}\int_{z_0+P}|f_m|^pe^{-m\phi}\right)\\
\leq&-\frac{1}{m}\frac{1}{\mu(P)}\int_{z_0+P}\log(|f_m|^pe^{-m\phi})\\
=&-\frac{1}{m}\frac{1}{\mu(P)}\int_{z_0+P}p\log|f_m|+
\frac{1}{\mu(P)}\int_{z_0+P}\phi.
\end{split}
\end{equation*}
Note that $p\log|f_m|$ is a plurisubharmonic function, and $f_m(z_0)=1$, the first term in the above equality is nonpositive,
therefore,
\begin{equation}\label{thm4 eq2}
-\frac{1}{m}\log\left(\frac{1}{\mu(P)}\int_{z_0+P}|f_m|^pe^{-m\phi}\right)\leq \frac{1}{\mu(P)}\int_{z_0+P}\phi.
\end{equation}
Combining inequalities \eqref{thm 4 eq1} and \eqref{thm4 eq2},
we get $$\phi(z_0)\leq \frac{\log C_m}{m}-\frac{\log\mu(P)}{m}+\frac{1}{\mu(P)}\int_{z_0+P}\phi.$$
As $\lim_{m\rightarrow\infty}\frac{\log C_m}{m}=0$ and $\lim_{m\rightarrow\infty}\frac{\log\mu(P)}{m}=0$,
taking limit, we have $$\phi(z_0)\leq \frac{1}{\mu(P)}\int_{z_0+P}\phi.$$
From Lemma \ref{psh}, $\phi$ is plurisubharmonic on $D$.

\end{proof}

%
%\bibliographystyle{abbrv}
%\bibliography{/Users/zhiwei/Paper/Bib/RefWang.bib}

\begin{thebibliography}{10}

\bibitem{Ber} B.~Berndtsson. \emph{$L^2$-methods for the $\bar\partial$-equation},  enotes, http://www.math.chalmers.se/$\sim$bob/not3.pdf.

\bibitem{Bob98}
B.~Berndtsson.
\newblock Prekopa's theorem and {K}iselman's minimum principle for
plurisubharmonic functions.
\newblock {\em Math. Ann.}, 312(4):785--792, 1998.

\bibitem{Bob06}
    B.~Berndtsson.
    \newblock Subharmonicity properties of the {B}ergman kernel and some other
    functions associated to pseudoconvex domains.
    \newblock {\em Ann. Inst. Fourier (Grenoble)}, 56(6):1633--1662, 2006.

\bibitem{Bob09}
    B.~Berndtsson.
    \newblock Curvature of vector bundles associated to holomorphic fibrations.
    \newblock {\em Ann. of Math. (2)}, 169(2):531--560, 2009.

\bibitem{Bl}
Z.~B\l ocki.
\newblock Suita conjecture and the {O}hsawa-{T}akegoshi extension theorem.
\newblock {\em Invent. Math.}, 193(1):149--158, 2013.

\bibitem{Dem92}
J.-P. Demailly.
\newblock Regularization of closed positive currents and intersection theory.
\newblock {\em J. Algebraic Geom.}, 1(3):361--409, 1992.

\bibitem{Dem12}
J.-P. Demailly.
\newblock {\em Analytic methods in algebraic geometry}, volume~1 of {\em
	Surveys of Modern Mathematics}.
\newblock International Press, Somerville, MA; Higher Education Press, Beijing,
2012.

\bibitem{DWZZ1}
F.~Deng, Z.~Wang, L.~Zhang, and X.~Zhou.
\newblock New characterization of plurisubharmonic functions and positivity of
direct image sheaves.
\newblock {\em arXiv:1809.10371}.

\bibitem{GZh12}
Q.~Guan and X.~Zhou.
\newblock Optimal constant problem in the {$L^2$} extension theorem.
\newblock {\em C. R. Math. Acad. Sci. Paris}, 350(15-16):753--756, 2012.

\bibitem{GZh15d}
Q.~Guan and X.~Zhou.
\newblock A solution of an {$L^2$} extension problem with an optimal estimate
and applications.
\newblock {\em Ann. of Math. (2)}, 181(3):1139--1208, 2015.

\bibitem{HPS16}
C.~Hacon, M.~Popa, and C.~Schnell.
\newblock Algebraic fiber spaces over abelian varieties: {A}round a recent
theorem by {C}ao and {P}\u aun.
\newblock In {\em Local and global methods in algebraic geometry}, volume 712
of {\em Contemp. Math.}, pages 143--195. Amer. Math. Soc., Providence, RI,
2018.

\bibitem{Hor65}
L.~H\"ormander.
\newblock {$L^{2}$} estimates and existence theorems for the {$\bar \partial
	$}\ operator.
\newblock {\em Acta Math.}, 113:89--152, 1965.

\bibitem{HI}
G.~Hosono and T.~Inayama.
\newblock A converse of H\"{o}rmander's $L^2$-estimate and new positivity
notions for vector bundles.
\newblock {\em arXiv:1901.02223v1.}

\bibitem{OT1}
T.~Ohsawa and K.~Takegoshi.
\newblock On the extension of {$L^2$} holomorphic functions.
\newblock {\em Math. Z.}, 195(2):197--204, 1987.

\end{thebibliography}

\end{document}